\documentclass[12pt]{amsart}

\usepackage{fullpage, graphicx, amsmath, amssymb}

\newtheorem{thm}{Theorem}[section]
\newtheorem{prop}[thm]{Proposition}

\newtheorem{lemma}[thm]{Lemma}
\newtheorem{cor}[thm]{Corollary}
\newtheorem{definition}[thm]{Definition}
\theoremstyle{definition}

\title{A Hopf algebraic approach to Schur function identities}

\author{Karen Yeats}
\thanks{The author would like to thank Stephanie van Willigenburg for useful conversations without which these results would not exist.}

\begin{document}
\maketitle

\begin{abstract}
Using cocommutativity of the Hopf algebra of symmetric functions, certain skew Schur functions are proved to be equal.  Some of these skew Schur function identities are new.
\end{abstract}

\section{Introduction}
The Schur symmetric functions are a particularly nice basis for symmetric functions.  Skew Schur functions are also very nice symmetric functions, but contain more mysteries.  The full set of relations between skew Schur functions is so far from known that even the question of when two skew Schur functions indexed by different skew shapes are equal is not fully understood.  

A characterization of when two skew Schur functions are equal is known in the case of skew ribbons \cite{BTvW}.  A partial but quite structural class of skew Schur function identities is given in \cite{mNvW}.  

This paper uses very different techniques to find equalities between skew Schur functions.  Some of the equalities obtainable with these new techniques are new.  However, many of the identities found in \cite{mNvW} are not obtainable with the results of this paper, so these new techniques complement the old ones rather than overtaking them.

The problem of skew Schur identities can also be seen as a special case of a question of Schur positivity.  A symmetric function is said to be Schur positive if the coefficients when written in terms of the basis of Schur functions are all nonnegative.  Symmetric function identities then form the equality case.  The special kind of identities considered above, equalities between two skew Schur functions, forms the equality case of the question of when the difference between two skew Schur functions is Schur positive.  Both Schur positivity and Schur positivity of differences specifically are questions which have had much study, see for example \cite{KWvW, LPPschur, mNnecess, mNquasi} and the references therein.

\section{Preliminaries and composition of shapes}

Given a partition $\lambda = (\lambda_1 \geq \lambda_2 \geq \cdots \geq \lambda_i)$ its \textbf{shape} (or Ferrers diagram or Young diagram)  is a set of left aligned rows of boxes with $\lambda_1$ boxes in the top row, $\lambda_2$ boxes in the second from the top row and so on.  By abuse of notation the shape of $\lambda$ will also be denoted $\lambda$.  Given another partition $\mu = (\mu_1 \geq \mu_2 \geq \cdots \geq \mu_j)$ if $j \leq i$ and $\mu_k \leq \lambda_k$ for $1\leq k \leq j$ then we say $\mu \leq \lambda$ and we can form the \textbf{skew shape} $\lambda/\mu$ by taking the boxes that appear in the shape of $\lambda$ but not $\mu$ when they are aligned with their top left corners coinciding.  A skew shape may or may not be connected.  The \textbf{size} of a shape is the number of boxes.  A \textbf{ribbon} is a shape that does not contain four boxes in a square, that is it does not contain the partition $(2,2)$ as a subshape.  \textbf{Young's lattice} is the poset of partition shapes ordered by the $\leq$ defined above.

Let $k$ be a commutative ring.  The \textbf{ring of symmetric functions} over $k$ is the subring of $k[[x_1, x_2, \ldots]]$ consisting of power series that are invariant under all permutations of the indeterminates and that are of bounded degree.  See \cite{Stv2} for details.

A \textbf{filling} of a shape (skew or not) is an assignment of positive integers to the boxes of the shape with the property that the integers are weakly increasing along the rows and strictly increasing along the columns.  To each filling we can associate a monomial in $k[[x_1, x_2, \ldots]]$ where the power of $x_i$ is the number of occurrences of $i$ in the filling.  Given a partition shape $\lambda$ the \textbf{Schur function} associated to $\lambda$ is the sum of all monomials associated to fillings of $\lambda$.  This turns out to be a symmetric function and in fact the Schur functions of partition shape form a basis for symmetric functions.  Given a skew shape $\lambda/\mu$ the associated \textbf{skew Schur function} is formed analogously.  Henceforth the term \emph{Schur function} will be used to refer to both the skew and non-skew cases.

If $\alpha$ and $\beta$ are shapes whose associated (skew) Schur functions are equal then we write
\[
\alpha \sim \beta
\]

Given a shape $\alpha$ (skew or not), rotating 180 degrees also gives a shape which is denoted $\alpha^*$.  It is a standard fact that $\alpha\sim \alpha^*$.  This can be seen by considering the relationship between the fillings in each case.  One trivial but 
useful consequence of this is that the Schur functions associated to 180 degree rotations of partition shapes also form a basis for symmetric functions (in fact the same basis).

Ribbon Schur functions are easier to understand than general skew Schur functions.  Two particular facts about ribbon Schur functions will be important in the following, see \cite{BTvW} for proofs.  First, ribbon Schur functions span symmetric functions.  Second the product of two ribbon Schur functions is easy to describe.  Specifically if $\alpha$ and $\beta$ are two ribbons, then the product of the Schur functions associated to $\alpha$ and $\beta$ is the sum of the ribbon Schur functions associated to the ribbon obtained by putting the leftmost box of $\beta$ immediately to the right of the rightmost box of $\alpha$ and to the ribbon obtained by putting the bottommost box of $\beta$ immediately above the topmost box of $\alpha$.

\medskip

We will need some more technical notation and definitions from \cite{mNvW}.  If we index the boxes of a shape by rows and columns, a \textbf{diagonal} of the shape is the set of boxes of the shape with the same difference between the row and column indices.

Given shapes $W$ and $\alpha$, say that $W$ lies in the top of $\alpha$ if $W$ appears as a connected subdiagram of $\alpha$ that includes
the northeasternmost box of $\alpha$.  Likewise $W$ lies in the bottom of $\alpha$ if it includes the southwesternmost box of $\alpha$.

Given two shapes $\alpha_1$ and $\alpha_2$ and a shape $W$ lying in the top of $\alpha_1$ and in the bottom of $\alpha_2$, the \textbf{amalgamation} of $\alpha_1$ and $\alpha_2$ along $W$, written $\alpha_1 \amalg_W \alpha_2$, is the shape obtained by placing $\alpha_1$ on top of $\alpha_2$ with the copies of $W$ identified.  This is Definition 3.5 of \cite{mNvW}.

A shape $\alpha$ is a  $W\rightarrow O \rightarrow W$ shape if
\begin{itemize}
  \item $W$ lies both in the top and in the bottom of $\alpha$,
  \item removing either $W$ leaves a connected shape,
  \item the subshape resulting from removing both $W$ shapes is $O$,
  \item there is at least one diagonal strictly between the two copies of $W$, and
  \item the southwesternmost box of $O$ has a box of the lower $W$ immediately to the west and the northeasternmost box of $O$ has a box of upper $W$ immediately to the east.
\end{itemize}
A shape $\alpha$ is a $W\uparrow O \uparrow W$ shape if it satisfies the first four properties above, and instead of the last property satisfies
\begin{itemize}
  \item the southwesternmost box of $O$ has a box of the lower $W$ immediately to the south and the northeasternmost box of $O$ has a box of upper $W$ immediately to the north.
\end{itemize}
Without loss of generality, McNamara and van Willigenburg also suppose that $W$ is maximal in the sense that no other shape contains $W$, satisfies the first two properties above, and occupies the same diagonals as $W$.
See pp 9,10 of \cite{mNvW} for details.

Their goal is to define a composition of shapes where copies of the second shape are put together according to the first shape.  To do that they need two operations to correspond to the left-right and up-down relation of boxes in the first shape.  These two operations are amalgamation, defined above, and the following operation, which is a sort of shifted amalgamation.

Suppose $\alpha$ is a $W\rightarrow O \rightarrow W$ shape and $\alpha_1$ and $\alpha_2$ are copies of $\alpha$.  Then $\alpha_1 \cdot_W \alpha_2$ is the shape where $\alpha_2$ is positioned so that its lower copy of $W$ is one position northwest of the upper copy of $W$ in $\alpha_1$.  Similarly if $\alpha$ is a $W\uparrow O \uparrow W$ shape then $\alpha_1 \cdot_W \alpha_2$ is the shape where $\alpha_2$ is positioned so that its lower copy of $W$ is one position southeast of the upper copy of $W$ in $\alpha_1$.

Now we are ready to define a modified version McNamara and van Willigenburg's composition.  
\begin{definition}[half of Definition 3.11 \cite{mNvW} modified]
Let $\alpha$ and $\beta$ be shapes and suppose $\beta$ is $W\rightarrow O \rightarrow W$ or $W\uparrow O \uparrow W$.  Then $\alpha \circ_W \beta$ (or $\alpha \circ \beta$ when $W$ is understood) is the shape defined as follows.  
\begin{itemize}
  \item There is a copy of $\beta$ for each box of $\alpha$, subject to some overlapping detailed below.  Write $\beta_b$ for the copy of $\beta$ corresponding to box $b$ of $\alpha$.
  \item If $b_1$ is a box of $\alpha$ immediately to the west of box $b_2$ then
    \begin{itemize}
    \item if $\beta$ is $W\rightarrow O\rightarrow W$ then  $\beta_{b_1}$ and $\beta_{b_2}$ overlap according to $\beta_{b_1} \amalg_W \beta_{b_2}$,
    \item if $\beta$ is $W\uparrow O \uparrow W$ then $\beta_{b_1}$ and $\beta_{b_2}$ overlap according to $\beta_{b_1} \cdot_W \beta_{b_2}$.
    \end{itemize}
  \item If $b_1$ is a box of $\alpha$ immediately to the south of box $b_2$ then
    \begin{itemize}
      \item if $\beta$ is $W\rightarrow O\rightarrow W$ then $\beta_{b_1}$ and $\beta_{b_2}$ overlap according to $\beta_{b_1} \cdot_W \beta_{b_2}$,
      \item if $\beta$ is $W\uparrow O \uparrow W$ then $\beta_{b_1}$ and $\beta_{b_2}$ overlap according to $\beta_{b_1} \amalg_W \beta_{b_2}$.
    \end{itemize}
\end{itemize}
\end{definition}

This differs from McNamara and van Willigenburg in that the roles of $\cdot_W$ and $\amalg_W$ have been swapped for $W\uparrow O \uparrow W$ shapes.  This means that for $W\uparrow O \uparrow W$ shapes $\beta$ what I write as $\alpha \circ \beta$ corresponds to what McNamara and van Willigenburg write as $\alpha^*\circ \beta$ and conversely. This makes no difference to the overall results since both in this paper and in their paper the ultimate results are about when $\alpha\circ\beta \sim \alpha^*\circ\beta$.  The definition given here is more convenient for the present purposes because then the role of the top and bottom key ribbons of Section~\ref{sec results} match in the $W\rightarrow O \rightarrow W$ cases and the $W\uparrow O \uparrow W$ cases rather than being swapped.  See Figure~\ref{fig all compositions} for a schematic of this composition.

\begin{figure}
  \includegraphics{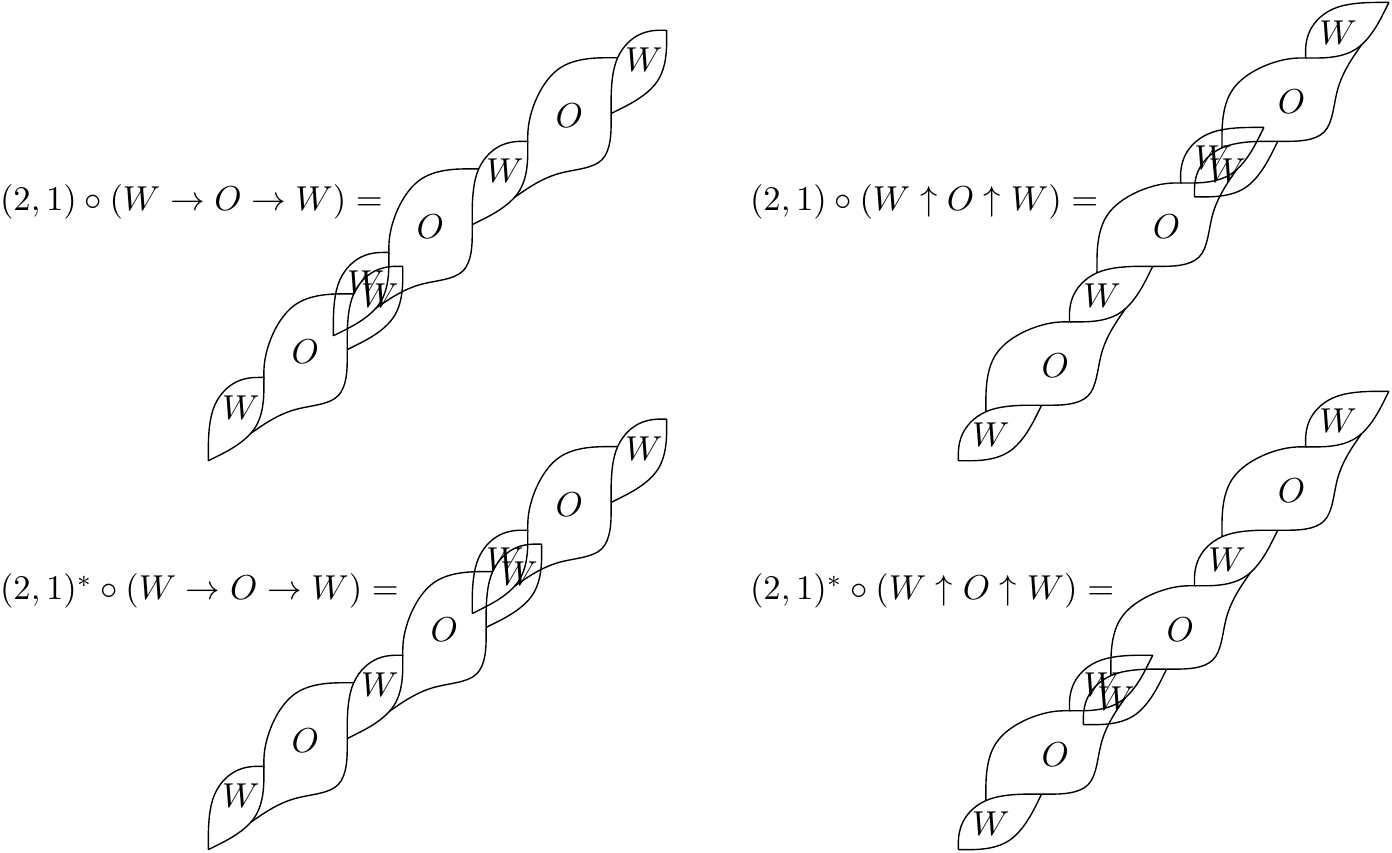}
  \caption{Schematic examples of the composition of shapes.}\label{fig all compositions}
\end{figure}

McNamara and van Willigenburg also define $W\uparrow O \rightarrow W$ and $W\rightarrow O \uparrow W$ shapes.  These shapes have a somewhat different feel.  The main results of \cite{mNvW} are cases where compositions involving the different $WOW$ shapes are $\sim$-equivalent.  The results of this paper are also of this form.  However, the techniques are quite different and it turns out that the class of identities that can be proved is also different.

\section{The shape Hopf algebra}
The approach of this paper relies crucially on the cocommutativity of the Hopf algebra of symmetric functions.

The Hopf algebra of symmetric functions is one of the classical combinatorial Hopf algebras \cite{GRhopf}.  If we think of symmetric functions in terms of Schur functions and then think of Schur functions in terms of the skew shapes which index them, then we can interpret the Hopf algebra of symmetric functions as a Hopf algebra operating on shapes.  

However, the full set of identities between Schur functions is not known and the Hopf algebra of symmetric functions viewed in terms of shapes also incorporates these identities which we do not know.
Consequently, it will be helpful to think of the Hopf algebra of symmetric functions as a quotient of a more naive Hopf algebra defined in terms of shapes but without such identities.  This is also a classical combinatorial Hopf algebra, namely it is the Hopf algebra of intervals in Young's lattice.  For a nice presentation of Hopf algebras of intervals in a poset see \cite{BMSvW} Theorem 2.1 and immediately before.  
For the present purposes, we are interested in taking the commutative version of this construction and in including skew shapes.  The resulting Hopf algebra can be defined directly as follows.

Let $\mathcal{P}$ be the set of all partition shapes.  Let $\mathcal{S}$ be the set of all shapes, both skew and non-skew, connected and not connected.  Let $\mathcal{C}$ be the set of connected shapes, skew and non-skew.  Let $k$ be the base field.\footnote{In fact as some authors such as Grinberg and Reiner \cite{GRhopf} have commented, the basic theory of combinatorial Hopf algebras also works over a commutative ring and often working over the ring $k=\mathbb{Z}$ is useful.}
Consider $k[\mathcal{C}]$.  If we view a disconnected skew shape as identified with the monomial of its connected components, then the multiplication in $k[\mathcal{C}]$ lives in a quotient of the vector space generated by $\mathcal{S}$, specifically the quotient given by identifying all disconnected skew shapes with the same multiset of connected components.  Note that this already captures one easy family of identities of skew Schur functions, namely if $\lambda$ is the disjoint union of $\mu$ and $\nu$ then $\lambda \sim \mu\nu$, that is skew Schur functions are multiplicative on disjoint unions.

Define the coproduct on shapes to match the coproduct of Schur functions, that is for a skew shape $\lambda/\mu$
\[
\Delta(\lambda/\mu) = \sum_{\mu \leq \eta \leq \lambda} \eta/\mu \otimes \lambda/\eta
\]
This defines $\Delta$ on $\mathcal{C}$, and so we define it on $k[\mathcal{C}]$ by extending as an algebra homomorphism.  
This is a sensible definition since this definition of 
$\Delta$ on disconnected shapes agrees with $\Delta$ on the product of the connected components.
Define the counit by  $\epsilon(\lambda/\mu) = 0$ provided $\lambda/\mu$ contains at least one box, $\epsilon(1) = 1$, and extend linearly.  These operations satisfy the axioms of a bialgebra, and the resulting bialgebra is graded and connected hence a Hopf algebra as well; see \cite{GRhopf} or some of the original combinatorial Hopf algebra work \cite{JRhopf, Santipode} for details.

Then $k[\mathcal{C}]$ with $\Delta$ and $\epsilon$ is the commutative version of the Hopf algebra of intervals of the poset of shapes (including skew shapes).  Call this Hopf algebra the \textbf{shape Hopf algebra}.  Note that the shape Hopf algebra is not cocommutative.

There is a map from the shape Hopf algebra to the symmetric function Hopf algebra given by taking each shape to its Schur function and extending linearly.  This map is an algebra homomorphism since the Schur function of a disconnected shape is the product of the Schur functions of the components and the map takes the element $1$ to $1$.  This map is a coalgebra homomorphism because the shape coproduct and counit were defined to match with those of the Schur functions.  Thus the map is a bialgebra homomorphism.  Any bialgebra map between two Hopf algebras is a Hopf algebra homomorphism, and so the map is a Hopf algebra homomorphism.  Further, the map is surjective onto symmetric functions and so by the Hopf algebra isomorphism theorems the symmetric function Hopf algebra is a quotient of the shape Hopf algebra.  In particular it is the quotient of the shape Hopf algebra given by forcing products to be what they are for Schur functions by imposing the required ideal of identities.  See \cite{eh} for similar quotients in quasi-symmetric functions.


\section{Results}\label{sec results}

Let $\beta$ be a partition shape that is a rectangle with the lower right corner box removed.  Let $\gamma$ be a $W\rightarrow O \rightarrow W$ or $W\uparrow O \uparrow W$ shape.  The main result is that $\beta \circ \gamma \sim \beta^* \circ \gamma$ (Theorem~\ref{main thm}) provided the ends of the composition do not have any extraneous ribbons (see Definition~\ref{def loose ends}).  The main tool is the cocommutativity of the Hopf algebra of symmetric functions.  

\begin{definition}
Given a $W\rightarrow O \rightarrow W$ or $W\uparrow O \uparrow W$ shape $\gamma$, the \textbf{top key ribbon} of $\gamma$ is the connected ribbon defined as follows
\begin{enumerate}
  \item If $\gamma$ is $W\rightarrow O \rightarrow W$ the top key ribbon is the part of the upper left ribbon of $\gamma \amalg_W \gamma$ from the first box after the $O$ of the first $\gamma$ up to and including the last box of the $O$ of the second $\gamma$.
  \item If $\gamma$ is $W\uparrow O \uparrow W$ the top key ribbon is the part of the upper left ribbon of $\gamma \amalg_W \gamma$ from the first box of the $O$ of the first $\gamma$ up to and including the last box not in the $O$ of the second $\gamma$.
\end{enumerate}
Define the \textbf{bottom key ribbon} of $\gamma$ to be the following connected ribbon
\begin{enumerate}
  \item If $\gamma$ is $W\uparrow O \uparrow W$ the bottom key ribbon is the part of the bottom right ribbon of $\gamma \amalg_W \gamma$ from the first box after the $O$ of the first $\gamma$ up to and including the last box of the $O$ of the second $\gamma$.
  \item If $\gamma$ is $W\rightarrow O \rightarrow W$ the bottom key ribbon is the part of the bottom right ribbon of $\gamma \amalg_W \gamma$ from the first box of the $O$ of the first $\gamma$ up to and including the last box not in the $O$ of the second $\gamma$.
\end{enumerate}
\end{definition}

The key ribbons are illustrated in Figure~\ref{key ribbons}.  See also Figure~\ref{eg fig} for how the key ribbons will appear in the main argument.

\begin{figure}
\includegraphics{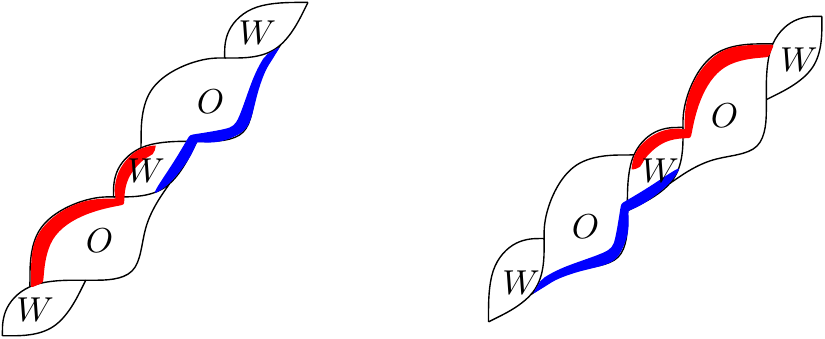}
\caption{Top key ribbons (red) and bottom key ribbons (blue) in $W\uparrow O\uparrow W$ and $W\rightarrow O \rightarrow W$ cases}\label{key ribbons}
\end{figure}

\begin{lemma}\label{size lemma}
The top key ribbon and bottom key ribbon of $\gamma$ have the same size, the same number of rows, and the same number of columns.
\end{lemma}

\begin{proof}
In $\gamma \amalg_W \gamma$, for both the top key ribbon and the bottom key ribbon, shifting one copy of $O$ by the key ribbon puts it onto the other copy of $O$.  Hence both key ribbons determine the same shift and so have the same number of rows and the same number of columns.  Since they are both ribbons they are also the same size.
\end{proof}

\begin{definition}
Given any shape $\lambda$, say a shape $\mu$ can be \textbf{taken out on the left} if the term of the form $\mu \otimes A$ in $\Delta(\lambda)$ is nonzero in the shape Hopf algebra.  Say $\mu$ can be \textbf{taken out on the right} if the term of the form $A \otimes \mu$ in $\Delta(\lambda)$ is nonzero in the shape Hopf algebra.
\end{definition}

To simplify the proof of Lemma~\ref{one key} it is worth making a digression into infinite shapes.  Extend the notion of taking out on the right and left to infinite shapes as follows.
Define an embedded infinite shape to be a set of boxes in the integer lattice $\mathbb{Z}^2$ such that every truncation given by restricting to $\{-N, \ldots, -1, 0, 1, \ldots, N\}^2$ is a shape (partition or skew).  The set of infinite shapes is then the set of embedded infinite shapes up to finite translation.  The restriction of an embedded infinite shape $\lambda$ to $\{-N, \ldots, -1, 0, 1\ldots, N\}^2$ will be denotes $\lambda_N$.
Say a shape $\mu$ can be taken out on the left of an infinite shape $\lambda$ if for any embedding of $\lambda$ there exists an embedding of $\mu$ in $\mathbb{Z}^2$ (independent of $N$) such that for all sufficiently large truncations this copy of $\mu$ gives a nonzero term of the form $\mu \otimes A$ in $\Delta(\lambda_N)$ in the shape Hopf algebra.  Taking out on the right in infinite shapes can be defined similarly.  Note that it is important that the embedding of $\mu$ be fixed for all $N$; is not sufficient to just require that $\mu$ can be taken out on the left in every truncation because there could be a different copy of $\mu$ on the border of each truncation but no $\mu$ coming out of the infinite shape.  For example $\mu = (1)$ cannot be taken out on the left or the right of the infinite vertical strip though it can be taken out of any truncation.

\begin{lemma}\label{lem infinite strip}
  Let $c$ be the infinite 1 column shape and let $d$ be the infinite 1 row shape.  Let $n$ be the size of the key ribbons of $\gamma$.  Then there are no connected ribbons of size $n$ which can be taken out on left nor which can be taken out on the right in either $c\circ \gamma$ or $d\circ \gamma$.
\end{lemma}

\begin{proof}
  Suppose $t$ is a ribbon of size $n$ which can be taken out on the right of $c\circ \gamma$.

  The infinite ribbon given by the upper left edge of $c\circ \gamma$ is periodic.  By the definition of the key ribbon, the period is $n$.  Since $t$ can be taken out on the right $t$ is a finite connected subribbon of this infinite ribbon.  Taking any truncation of $c\circ \gamma$ which includes at least one $\gamma$ on each side of $t$ does not affect whether or not $t$ comes out.   Since $|t| =n$ this means that the last box before $t$ and the first box in $t$ are in the same local situation as the last box of $t$ and the first box not in $t$.  Since $t$ can be taken out, its first box must be above the box before it and so the last box of $t$ is trapped by the box above it contradicting the fact that it can be taken out.

  The argument is analogous for the other cases.
\end{proof}

\begin{lemma}\label{lem infinite one key}
  Let $h$ be an infinite shape such that every truncation is a partition shape and such that the numbers of rows in the truncations are unbounded and the numbers of columns in the truncations are unbounded.
  \begin{enumerate}
  \item Let $r$ be the top key ribbon of $\gamma$.
The only connected ribbon of size $|r|$ that can be taken out on the left of $h \circ \gamma$ is the copy of $r$ that appears in the copy of $\gamma$ corresponding to the top left box of $h$. 
\item Let $s$ be the bottom key ribbon of $\gamma$.  The only connected ribbon of size $|s|$ that can be taken out on the right of $h^* \circ \gamma$ is the copy of $s$ that appears in the copy of $\gamma$ corresponding to the bottom right box of $h^*$. 
\end{enumerate}
\end{lemma}

\begin{proof}
First note that $h$ consists of a box, a 1 way infinite column beneath the box, a 1 way infinite row after the box, and potentially some boxes to the lower right of the boxes already given which locally maintain a partition shape.

  Suppose $t$ is a ribbon of size $|r|$ that can be taken out on the left of $h \circ \gamma$.

  The infinite ribbon given by the upper left edge of $h\circ \gamma$ is made of two periodic pieces; the upper left edges of the left column of $h$ composed with $\gamma$ and the upper left edges of the top row of $h$ composed with $\gamma$.  Boxes to the lower right in $h$ do not contribute to this ribbon.  Since $t$ can be taken out on the left $t$ is a finite connected subribbon of this infinite ribbon.

  By the definition of the key ribbon, the first box that is not part of the column periodicity is the rightmost box of the top key ribbon of the copy of $\gamma$ corresponding to the top left box of $h$.  Similarly, the last box that is not part of the row periodicity is the leftmost box of the top key ribbon of the copy of $\gamma$ corresponding to the top left box of $h$.  See Figure~\ref{fig end of per} for an example.

\begin{figure}
  \includegraphics{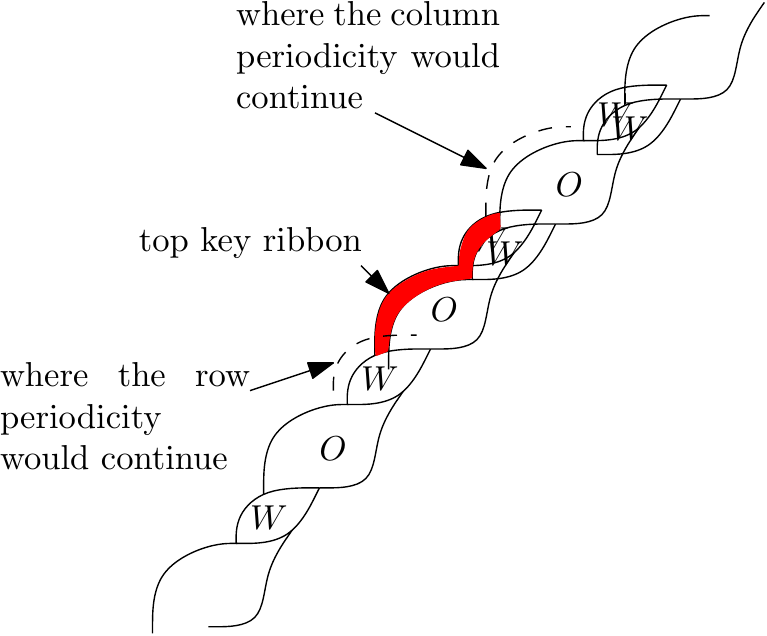}
  \caption{A schematic of the area around the copy of $\gamma$ corresponding to the top left box.}
  \label{fig end of per}
\end{figure}

  As in Lemma~\ref{lem infinite strip} if $t$ is within either periodic part of the infinite ribbon then it cannot be taken out.  Thus, by the boundaries of the previous paragraph, $t$ must include the rightmost box and the leftmost box of the top key ribbon of the copy of $\gamma$ corresponding to the top left box of $h$.  Thus $t$ must be the top key ribbon of the copy of $\gamma$ corresponding to the top left box of $h$.

  The result for $h^*$ is analogous.
\end{proof}

We need one last technical restriction to characterize the class of shapes $\gamma$ which we can use.
\begin{definition}\label{def loose ends}
  \begin{enumerate}
  \item  Let $\gamma$ be a $W\rightarrow O \rightarrow W$ shape and let $n$ be the size of its key ribbons.
    Say that $\gamma$ has \textbf{no loose end ribbons} if no ribbon of length $n$ which begins to the left of the top key ribbon can be taken out on the left and no ribbon of length $n$ which ends to the right of the bottom key ribbon can be taken out on the right.
  \item Let $\gamma$ be a $W\uparrow O \uparrow W$ shape and let $n$ be the size of its key ribbons.
    Say that $\gamma$ has \textbf{no loose end ribbons} if no ribbon of length $n$ which ends to the right of the top key ribbon can be taken out on the left and no ribbon of length $n$ which begins to the left of the bottom key ribbon can be taken out on the right.
  \end{enumerate}
\end{definition}

\begin{lemma}\label{one key}
Let $\lambda$ be a partition shape and let $\gamma$ have no loose end ribbons.
\begin{enumerate}
\item Let $r$ be the top key ribbon of $\gamma$.  
Then, the only connected ribbon of size $|r|$ that can be taken out on the left of $\lambda \circ \gamma$ is $r$.  $r$ can be taken out exactly once and this is when it appears in the copy of $\gamma$ corresponding to the top left box of $\lambda$. 
\item Let $s$ be the bottom key ribbon of $\gamma$. 
  The only connected ribbon of size $|s|$ that can be taken out on the right of $\lambda^* \circ \gamma$ is $s$.  $s$ can be taken out exactly once and this is when it appears in the copy of $\gamma$ corresponding to the bottom right box of $\lambda^*$. 
\end{enumerate}
\end{lemma}

\begin{proof}
  Suppose $t$ is a ribbon of size $|r|$ that can be taken out on the left of $\lambda \circ \gamma$.

  Consider the upper left ribbon of $\lambda \circ \gamma$.  This ribbon is made of two pieces: first a finite section of the upper left ribbon of $h\circ \gamma$ from Lemma~\ref{lem infinite one key}, including the part corresponding to the upper left corner; and second, an extra bit.  If $\gamma$ is $W\rightarrow O \rightarrow W$ then this extra bit is on the left and comes from the part of the copy of $\gamma$ coming from the lower left box of $\lambda$ which is to the left of the top key ribbon.  If $\gamma$ is $W\uparrow O \uparrow W$ then this extra bit is on the right and comes from the part of the copy of $\gamma$ coming from the upper right box of $\lambda$ which is to the right of the top key ribbon.  See Figure~\ref{fig extra bit} for an example of the extra bit.

  By the assumption of no loose end ribbons $t$ cannot contain any boxes of this extra bit.  Then, by the arguments of Lemma~\ref{lem infinite one key}, $t$ must be the top key ribbon in the copy of $\gamma$ corresponding to the top left box of $\lambda$.  

The proof for $s$ is analogous.
\end{proof}

\begin{figure}
  \includegraphics{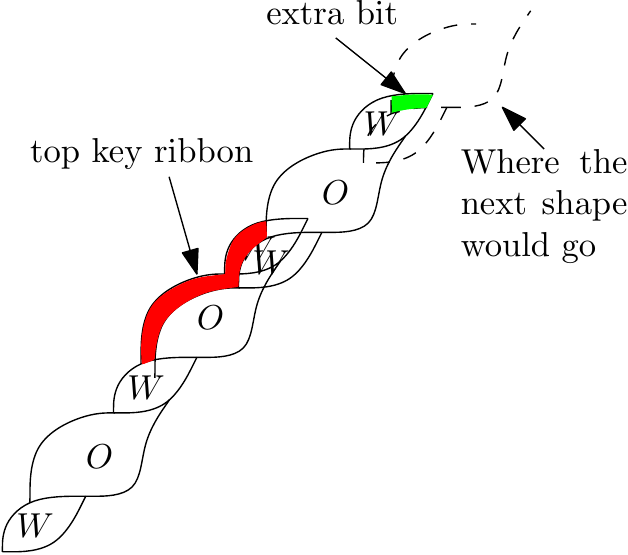}
  \caption{An example of the extra bit.}
  \label{fig extra bit}
\end{figure}

Given a basis for symmetric functions of a given degree it will be useful to have a notation for the coefficient vectors in this basis.  

\begin{definition}
  Let $\mathcal{B}$ be an ordered basis for symmetric functions consisting of Schur functions.  Let $\lambda$ be a shape.  Let $\overline{\lambda}_{\mathcal{B}}$ be the coefficient vector (written as a row) of the Schur function indexed by $\lambda$ relative to $\mathcal{B}$.
\end{definition}

Whenever convenient and without comment, when dealing with shapes of a particular size $n$, $\mathcal{B}$ and $\overline{\lambda}_{\mathcal{B}}$ will be viewed as restricted to size $n$ and hence finite.

\begin{lemma}\label{signed sum}
  Let $\mathcal{B}$ be an ordered basis for symmetric functions made up of connected ribbon Schur functions and let $\lambda$ be a shape that is not a connected ribbon.  Let $v$ be the vector of length $|\mathcal{B}|$ that is $1$ if the corresponding entry of $\mathcal{B}$ has an even number of rows and is $-1$ if the corresponding entry of $\mathcal{B}$ has an odd number of rows.  Then $v \cdot \overline{\lambda}_{\mathcal{B}} = 0$.
\end{lemma}

\begin{proof}
There are two cases to consider, disconnected ribbons and non-ribbons (both connected and disconnected).

Consider a ribbon.  The product of two connected ribbons creates two connected ribbons whose number of rows differs by $1$ so one has an odd number of rows and one has an even number of rows.  Thus, using the equivalences in the Hopf algebra to convert any disconnected ribbon into connected ribbons gives a sum of ribbons with exactly half having an even number of rows and half having an odd number of rows.  Furthermore, all linear relations among ribbons are generated in this way (see Proposition 2.2 of \cite{BTvW}).  So, writing any given connected ribbon in terms of a fixed basis of connected ribbons does not change the difference between the number (weighed by their coefficients) of ribbons in the sum with an odd number of rows and the number with an even number of rows.  This gives the desired result in the case of disconnected ribbons.

\medskip

Consider a non-ribbon.  
Given a connected shape containing at least one $(2,2)$, say the boxes above or to the right of the upper rightmost $(2,2)$ are the right ribbon and say the boxes below or to the left of the lower leftmost $(2,2)$ are the left ribbon.
This proof is a triple induction, first on the number of $(2,2)$s inside the shape, second within connected shapes with the same number of $(2,2)$, by sum of the sizes of the left and right ribbons, and third within shapes with the same total size of left and right ribbon, by the size of the left ribbon.  The base case for the outer induction is the ribbons, connected and disconnected.  

Given a non-ribbon shape with $k$ $(2,2)$s.  If the shape is connected and has no right or left ribbons, add a box to the end of the first row.  By cocommutativity the results of taking out one box on the right in all possible ways must equal the results of taking out one box on the left in all possible ways.  One of the ways of taking out one box on the right gives the shape itself; all other ways of taking out one box on either the left or the right either remove a $(2,2)$ or disconnect the shape.  Since the shape had no right or left ribbons, in the case where the shape becomes disconnected, each connected component must contain at least one $(2,2)$ and so each component has fewer than $k$ $(2,2)$s.   To summarize, cocommutativity gives an equality between two sums of shapes.  One of the shapes in one of the sums is the shape we began with while the others have fewer $(2,2)$s.  Solving for the shape itself, we obtain the shape itself as a sum of other shapes where each connected component has fewer $(2,2)$s.  Notice that any time a connected component has only one $(2,2)$ then this $(2,2)$ can be broken in exactly two ways: one by taking out a box on the right and one by taking out a box on the left and both ways give ribbons with the same number of rows.  Thus when the sum gives ribbons they will appear in pairs with opposite sign and same row parity.  By induction and the result for products of ribbons, the result holds in the case of shapes with $k$ $(2,2)$s and no left or right ribbons.

Now consider a connected shape with $k$ $(2,2)$s and no left ribbon.  Again add a box to the end of the first row.  Again, by cocommutativity, the results of taking out one box on the right in all possible ways must equal the results of taking out one box on the left in all possible ways, and one of the results of taking out one box on the right is the shape itself.  Taking out any other single box either removes a $(2,2)$, which is fine by induction, or disconnects the shape.  If disconnecting the shape results in components each with at least one $(2,2)$ then again we're fine by induction.  The remaining possibility is that the box removed was in the right ribbon which results in a connected component with no left ribbon and a smaller right ribbon.  Solving for the shape itself, similarly to the previous paragraph, this case holds by induction.

Finally take any shape with $k$ $(2,2)$s.  Assume the result holds for any shape with fewer $(2,2)$s (outer induction), for any shape with smaller total size of left and right ribbons (middle induction), and for any shape with the same total size of left and right ribbons but with smaller left ribbon (inner induction).  The set up is the same as before.  Call the connected component that contains the top row the top component.  Again one way of taking out one box on the right is the shape itself.  Taking out any other single box does one of the following:
\begin{itemize}
  \item removes a $(2,2)$; the result holds on this piece by the outer induction
  \item removes a box from the left or right ribbon of a component that is not the top component; the result holds on this piece by the middle induction
  \item removes a box from the left ribbon of the top component; the result holds on this piece by the inner induction
  \item creates a new connected component; the result holds on this piece by the outer induction if each new component contains a $(2,2)$ and by the middle induction otherwise.
\end{itemize}
Solving for the shape itself completes the proof.
\end{proof}

\begin{lemma}\label{only multiple is 1}
  Let $s$ and $t$ be connected ribbon Schur functions with the same number of rows.  If $s$ is a nonzero scalar multiple of $t$ then $s=t$ 
\end{lemma}

\begin{proof}
Take a basis of connected ribbon Schur functions that includes $s$.  As noted in the proof of Lemma \ref{signed sum}, 
writing any given connected ribbon in terms of a fixed basis of connected ribbons does not change the coefficient-weighted difference between the number of ribbons in the sum with an odd number of rows and the number with an even number of rows.  Writing $t$ in terms of the basis the only ribbon in this sum is $s$ and both $s$ and $t$ have the same number of rows so the weight must be $1$.
\end{proof}

\begin{figure}
\includegraphics{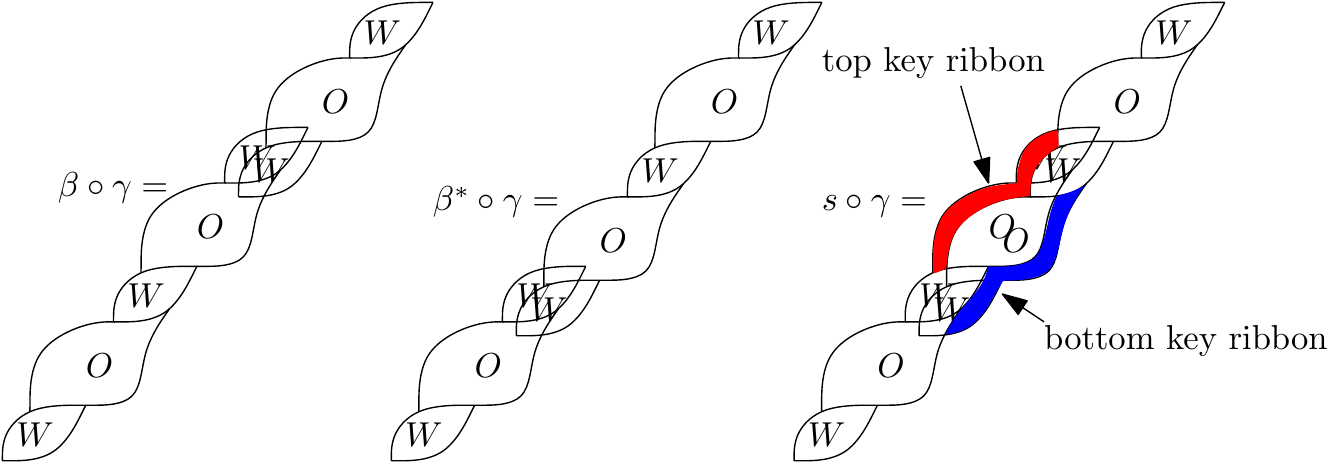}
\caption{An example of the position of the key ribbons in the set up of Proposition~\ref{same key prop}}\label{eg fig}
\end{figure}

\begin{prop}\label{same key prop}
  Let $\beta$ be a partition shape that is a rectangle with the lower right corner box removed.  Let $\gamma$ be a $W\rightarrow O \rightarrow W$ or $W\uparrow O \uparrow W$ shape with no loose end ribbons.  Suppose the top key ribbon and the bottom key ribbon of $\gamma$ are the same.  Then $\beta \circ \gamma \sim \beta^* \circ \gamma$. 
\end{prop}

\begin{proof}
Let $s$ be $\beta$ with the lower right corner filled.  Note that if we lay $\beta \circ \gamma$ and $\beta^* \circ \gamma$ on top of each other with the copies of $\gamma$ corresponding to common boxes of $\beta$ and $\beta^*$ identified then the resulting shape is $s\circ \gamma$.  Furthermore, those boxes of $s\circ \gamma$ that are not in $\beta \circ \gamma$ are the unique copy of top key ribbon of $\gamma$ that can be removed from $s\circ \gamma$ on the left as described in Lemma~\ref{one key}.  Symmetrically, those boxes of $s\circ \gamma$ that are not in $\beta^* \circ \gamma$ are the unique copy of bottom key ribbon of $\gamma$ that can be removed from $s\circ \gamma$.  Let $\alpha$ be the key ribbon of $\gamma$.  See Figure~\ref{eg fig} for an illustration for $\beta = (2,1)$ and $\gamma = W\uparrow O \uparrow W$.

Consider $\Delta(s \circ \gamma)$.  Working in the shape Hopf algebra, we see that in $\Delta(s \circ \gamma)$ we have terms $\alpha \otimes (\beta^* \circ \gamma)$ and $(\beta \circ \gamma) \otimes \alpha$, and by Lemma~\ref{one key} no other terms are of the form $\alpha \otimes A$ or $A \otimes \alpha$.  
We now need some notation.  Given a shape $\lambda$ let $\lambda \otimes R_\lambda$ be the sum of all terms of the form $\lambda\otimes A$ in $\Delta(s \circ \gamma)$ and let $L_\lambda \otimes \lambda$ be the sum of all terms of the form $A \otimes \lambda$ in $\Delta(s\circ \gamma)$.

\medskip

Moving to the symmetric functions Hopf algebra for the rest of the proof, let $\mathcal{B}$ be a basis for symmetric functions of degree $|\alpha|$ that includes the Schur function indexed by $\alpha$ and for which all other elements of the basis are also connected ribbon Schur functions.  This is possible since the connected ribbon Schur functions span symmetric functions.

Build two matrices $L$ and $R$ with entries in the symmetric function Hopf algebra.  $L$ and $R$ will each have rows indexed by the shapes of size $|\alpha|$ that are not connected ribbons and columns indexed by the elements of $\mathcal{B}$.  Given a shape $\lambda$ that is not a connected ribbon, the row corresponding to $\lambda$ in $L$ is $L_\lambda \overline{\lambda}_{\mathcal{B}}$ in $R$ is $R_\lambda\overline{\lambda}_{\mathcal{B}}$.  Note that $L_\lambda$ and $R_\lambda$ are linear combinations of shapes, so are scalars in this context, while $\overline{\lambda}_{\mathcal{B}}$ is a rational vector. 

By construction the row of $R$ corresponding to $\lambda$ gives the contributions of $\lambda$ to terms of the form $b\otimes A$ with $b\in \mathcal{B}$ in $\Delta(s \circ \gamma)$.  The same holds for $L$ with $b\otimes A$ replaced by $A\otimes b$.  Given $b\in \mathcal{B}$, $b\neq \alpha$, by Lemma~\ref{one key} there is no contribution to the $b\otimes A$ terms of $\Delta(s \circ \gamma)$ from $b$ itself or from any other connected ribbon.  Therefore the full contribution of the form $b\otimes A$ is given by the sum of the entries of the $b$ column of $R$.  Similarly the full contribution of the form $A\otimes b$ is given by the sum of the entries of the $b$ column of $L$.  By cocommutativity of the symmetric function Hopf algebra the sum of the entries of the $b$ column of $R$ is equal to the sum of the entries of the $b$ column of $L$ for $b \in \mathcal{B}$, $b\neq \alpha$.

By Lemma~\ref{signed sum} the column of $L$ corresponding to $\alpha$ is a signed sum of the other columns $L$ and likewise, with the identical linear combination, for $R$.  Therefore, using the observation of the previous paragraph, the sum of the entries of the $\alpha$ column of $R$ is equal to the sum of the entries of the $\alpha$ column of $L$.

The full contributions to terms involving $\alpha$ in $\Delta(s\circ \gamma)$ must also include the direct contribution from taking out $\alpha$ on the top and bottom which was calculated earlier in the proof.  Taking everything together, the full contribution to the terms of the form $\alpha \otimes A$ in $\Delta(s \circ \gamma)$ is the sum of the $\alpha$ column of $R$ plus $\beta^* \circ \gamma$ and the full contribution to the terms of the form $A \otimes \alpha$ in $\Delta(s \circ \gamma)$ is the sum of the $\alpha$ column of $L$ plus $\beta \circ \gamma$.  Therefore by cocommutativity $\beta^* \circ \gamma$ and $\beta \circ \gamma$ are equal as symmetric functions.  That is, $\beta^* \circ \gamma \sim \beta \circ \gamma$.

\end{proof}

\begin{thm}\label{main thm}
  Let $\beta$ be a partition shape that is a rectangle with the lower right corner box removed.  Let $\gamma$ be a $W\rightarrow O \rightarrow W$ or $W\uparrow O \uparrow W$ shape with no loose end ribbons.  
Then $\beta \circ \gamma \sim \beta^* \circ \gamma$. 
\end{thm}

\begin{proof}
Let $\alpha_1$ be the top key ribbon of $\gamma$ and let $\alpha_2$ be the bottom key ribbon of $\gamma$.  If the Schur functions indexed by $\alpha_1$ and $\alpha_2$ are not linearly independent then by Lemma~\ref{only multiple is 1} they are equal.  In that case we can argue as in the proof of Proposition~\ref{same key prop}.

Now assume the Schur functions indexed by $\alpha_1$ and $\alpha_2$ are linearly independent.  Let $\mathcal{B}$ be a basis for symmetric functions of degree $|\alpha|$ that includes the Schur functions indexed by $\alpha_1$ and $\alpha_2$ and for which all other elements of the basis are also connected ribbon Schur functions.  Build the matrices $L$ and $R$ as in the proof of Proposition~\ref{same key prop}.

$\mathcal{B}$ is not the correct basis for completing this argument because the top and bottom key ribbons are contributing to different basis elements hence to different columns.  In order to fix this we will modify $\mathcal{B}$ in order to form a new basis $\mathcal{B}'$.  Once that has been done, the rest of the argument will proceed as for Proposition~\ref{same key prop} using $\mathcal{B}'$ in place of the original $\mathcal{B}$.

Let $\mathcal{B}'$ be the basis formed by removing $\alpha_2$ from $\mathcal{B}$ and replacing it by $\alpha_2-\alpha_1$.  Let $L'$ and $R'$ be the corresponding matrices for $\mathcal{B}'$.  The columns other than the column for $\alpha_2$ (now $\alpha_2-\alpha_1$) are unchanged.  So for any $b\in \mathcal{B}'$ with $b\neq \alpha_1$ and $b\neq \alpha_2-\alpha_1$, by Lemma~\ref{one key} and cocommutativity the sum of the entries of the $b$ column of $R'$ is equal to the sum of the entries of the $b$ column of $L'$.

Let $v$ be the vector of length $|\mathcal{B}|$ that is $1$ if the corresponding entry of $\mathcal{B}$ has an even number of rows and is $-1$ if the corresponding entry of $\mathcal{B}$ has an odd number of rows.  Let $v'$ be $v$ with the entry corresponding to $\alpha_2$ ($\alpha_2-\alpha_1$ in $\mathcal{B}'$) set to $0$.  Take any shape $\lambda$ of size $|\alpha_1|$.  The $\alpha_1$ entry of $\overline{\lambda}_{\mathcal{B}'}$ is the sum of the $\alpha_1$ and $\alpha_2$ entries of $\overline{\lambda}_\mathcal{B}$ while the $\alpha_2-\alpha_1$ entry of $\overline{\lambda}_{\mathcal{B}'}$ is the difference between the $\alpha_1$ and $\alpha_2$ entries of $\overline{\lambda}_\mathcal{B}$.  By Lemma~\ref{size lemma} $\alpha_1$ and $\alpha_2$ have the same weight in $v$ and so Lemma~\ref{signed sum} implies that $v' \cdot \overline{\lambda}_{\mathcal{B}'} = 0$.

Consequently, the column of $L'$ corresponding to $\alpha_1$ is a signed sum of the other columns of $L'$ which does not use the $\alpha_2-\alpha_1$ column and likewise for $R'$.  Therefore the sum of the entries of the $\alpha_1$ column of $R'$ is equal to the sum of the entries of the $\alpha_1$ column of $L'$.

Finally, $\alpha_1$ taken out on the left contributes $\alpha_1 \otimes (\beta^* \circ \gamma)$ and $\alpha_2$ taken out on the right contributes $(\beta \circ \gamma) \otimes (\alpha_2 - \alpha_1) + (\beta \circ \gamma) \otimes \alpha_1$.  Considering the $\alpha_1$ contributions on both sides of the tensor, by the previous paragraph and cocommutativity, as in the proof of Proposition~\ref{same key prop}, we get that $\beta^* \circ \gamma$ and $\beta \circ \gamma$ are equal as symmetric functions.  That is, $\beta^* \circ \gamma \sim \beta \circ \gamma$.
\end{proof}

Note that by \cite{mNvW} Lemma 3.19 ii (which is just an elementary check) $(\beta \circ_W \gamma)^* = \beta^* \circ_{W^*} \gamma^*$ and ${}^*$ swaps $W\rightarrow O \rightarrow W$ and $W\uparrow O\uparrow W$ shapes.  So since rotating a shape 180 degrees does not affect the Schur function, applying ${}^*$ to $\beta^* \circ \gamma$ in Theorem~\ref{main thm} gives the following corollary.
\begin{cor}
Let $\beta$ be a partition shape that is a rectangle with the lower right corner box removed.  Let $\gamma$ be a $W\rightarrow O \rightarrow W$ or $W\uparrow O \uparrow W$ shape with no loose end ribbons.  Then $\beta \circ \gamma \sim \beta \circ \gamma^*$.
\end{cor}

\section{Discussion}

These results are interesting for two reasons.  First of all the tools they use are quite different from what has been used for similar results (see \cite{vWeq, BTvW, mNvW}).  The tools of this paper don't require heavy algebraic machinery, using primarily just the fact of cocommutativity of the symmetric function Hopf algebra.

Secondly, although there are many identities obtainable with other methods which are not obtainable by the present methods, other identities are obtainable by the present methods but not by other methods.
Most interesting along these lines is to compare Theorem~\ref{main thm} with the main results of \cite{mNvW}.  Those authors are able to tackle a much larger class of shapes since they do not need to restrict the shape of $\beta$ nearly as heavily.  However, they do require a technical hypothesis, called hypothesis V, which they conjecture to be unnecessary.  They give an example which they cannot obtain by their methods because of hypothesis V, illustrated here in Figure~\ref{end eg fig}.  This example follows from Theorem~\ref{main thm} with $\beta=(2,1)$ and $\gamma=(4,4,2,2)/(2,1) = (1,1) \rightarrow (2,2,1,1)/(1) \rightarrow (1,1)$.  Hence it gives an example obtainable by the methods of this paper, but not by the methods of \cite{mNvW}.

\begin{figure}
\includegraphics{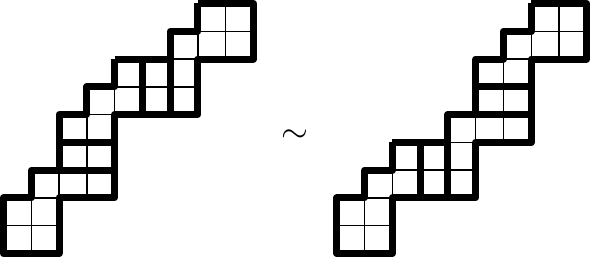}
\caption{An example of McNamara and van Willigenburg which can be done by the methods of this paper but not by the methods of \cite{mNvW}.}\label{end eg fig}
\end{figure}

Finally, it is natural to ask whether the present methods can be generalized.  One question is whether the restrictions on $\beta$ can be lessened.  The simplicity of the key ribbons depends on the special shape of $\beta$.  The difficulty in generalizing is in finding an appropriate generalization of Lemma~\ref{one key}; once we go beyond key ribbons, there can be more than one shape of the same size that can be taken out.

The other unpleasant hypothesis was the hypothesis of no loose end ribbons.  Some hypothesis is required in place of the no loose end ribbon hypothesis.  To see this, take $\beta = (2,1)$ and $\gamma = (8,7,2)/(3,1)$, then $\gamma$ is a $W\rightarrow O \rightarrow W$ shape,\footnote{Thanks to the referee for this example.} however $\beta\circ\gamma \not\sim \beta^*\circ \gamma$.  The key ribbons have size 6 but there is another ribbon of size 6 which can be removed from $\gamma$ to the left of the upper key ribbon and so this shape has a loose end ribbon.  Ideally, the no loose end ribbon hypothesis would be replaced by something less obscure.
Note that it also makes sense in view of the results of McNamara and van Willigenburg that some additional hypothesis should be necessary because their hypothesis IV has not been incorporated in this Hopf setup.  Hypothesis IV is what fails for McNamara and van Willigenburg in the previous example.


\bibliographystyle{plain}
\bibliography{main}

\end{document}